\documentclass[reqno,10pt]{amsart}
\usepackage{amsmath,amsthm,amsfonts,amssymb,latexsym}
\usepackage{hyperref,times,enumitem}

\newtheorem{theorem}{Theorem}[section]
\newtheorem{lemma}{Lemma}[section]

\theoremstyle{remark}
\newtheorem{remark}{Remark}
\numberwithin{equation}{section}

\newcommand{\infsum}{\sum^\infty_{n=-\infty}}
\newcommand{\cfplus}{\mathbin{\genfrac{}{}{0pt}{}{}{+}}}

\newcommand{\lb}{\left}
\newcommand{\rb}{\right}
\begin{document}
\title[$P$--$Q$ ``mixed'' modular equations of degree 15]
{$P$--$Q$ ``mixed'' modular equations of degree 15}
\author[S. Chandankumar and B. Hemanthkumar]
{S. Chandankumar and B. Hemanthkumar}
\address[S. Chandankumar and B. Hemanthkumar]{Department of Mathematics, M. S. Ramaiah University of Applied Sciences, Peenya campus,
Bengaluru - 560058, Karnataka,India}\email{{hemanthkumarb.30@gmail.com, chandan.s17@gmail.com}}
\begin{abstract}
Ramanujan in his second notebook recorded total of seven $P$--$Q$ modular equations involving theta--function $f(-q)$ with moduli of orders 1, 3, 5 and 15. In this paper, modular equations analogous to those recorded by Ramanujan are obtained involving his theta--functions $\varphi(q)$ and $\psi(-q)$ with moduli of orders 1, 3, 5 and 15. As a consequence, several values of quotients of theta--function and a continued fraction of order 12 are explicitly evaluated.  
\end{abstract}
\maketitle
\section{Introduction}
\noindent Throughout, we assume that $|q|<1$, Ramanujan's general theta--function $f(a,b)$ is defined by
\begin{equation}
f(a,b):=\sum_{n=-\infty}^{\infty}a^{n(n+1)/2}
b^{n(n-1)/2},\,\,\,|ab|<1.
\end{equation}
  Furthermore following Ramanujan we define three special cases of $f(a,b)$:
  \begin{equation*}\begin{split}
&\varphi(q):=f(q,q)=\infsum{q^{n^2}},\\
&\psi(q):=f(q,q^3)=\sum_{n=0}^{\infty}q^{n(n+1)/2},\\
&f(-q):=f(-q,-q^2)=\sum_{n=-\infty}^{\infty}(-1)^nq^{n(3n-1)/2}.
\end{split}\end{equation*}
The ordinary or Gaussian hypergeometric function is defined by
$$_2F_1(a,b;c;z):=\sum_{n=0}^{\infty}\frac{\lb(a\rb)_n \lb(b\rb)_n}{\lb(c\rb)_n n!}z^n,\ \ \ 0\leq|z|<1,$$
where $a$, $b$, $c$ are complex numbers, $c\neq0,-1,-2,\ldots$, and $$(a)_0=1,\ \ (a)_n=a(a+1)\cdots(a+n-1)\ \ \textrm{for any positive integer}\ \ n.$$
Now we recall the notion of a ``mixed'' modular equation. Let $K(k)$ be the complete elliptic integral of the
first kind of modulus $k$, we have
\begin{equation}\label{ee11}
K(k):=\int_0^{\frac{\pi}{2}}\frac{d\phi}{\sqrt{1-k^2\sin^2\phi}}
=\frac{\pi}{2}\sum_{n=0}^{\infty}\frac{\left(\frac{1}{2}\right)^2_n}{\left(n!\right)^2}k^{2n}=\frac{\pi}{2}\varphi^2(q)
,\,\,\,\,\,(0<k<1),
\end{equation}
and set $K'=K(k')$, where $k'=\sqrt{1-k^2}$  is called
complementary modulus of $k$. It is classical to set $q(k)=e^{-\pi
K(k')/K(k)}$ so that $q$ is one-to-one and increases from 0 to 1.

 Following Ramanujan we set $\alpha:=k^2$, and let $K$, $K'$, $L_1$, $L_1'$, $L_2$, $L_2'$, $L_3$ and $L_3'$ denote complete elliptic integrals of the first kind corresponding, in pairs, to the moduli $\sqrt{\alpha}$, $\sqrt{\beta}$, $\sqrt{\gamma}$ and $\sqrt{\delta}$, and their complementary moduli, respectively. Suppose that the equalities
\begin{equation}\label{ee14}
n_1\frac{K'}{K}=\frac{L_1'}{L_1},\ \
n_2\frac{K'}{K}=\frac{L_2'}{L_2} \ \ \textrm {and}\ \
n_3\frac{K'}{K}=\frac{L_3'}{L_3},
\end{equation}
hold for some positive integers $n_1$, $n_2$ and $n_3$. Then the relation between the moduli $\sqrt{\alpha}$, $\sqrt{\beta}$, $\sqrt{\gamma}$ and $\sqrt{\delta}$ that is induced by \eqref{ee14} is called as a ``mixed'' modular equation of composite degree $n_3=n_1n_2$.  We say that $\beta$, $\gamma$ and $\delta$ are of degrees $n_1$, $n_2$ and $n_3$ respectively over $\alpha$.  The multipliers $m=K/L_1$ and $m'=L_2/L_3$ are algebraic relations involving $\alpha$, $\beta$, $\gamma$ and $\delta$.

 Ramanujan \cite{SR2} recorded total of seven $P$--$Q$ modular equations involving theta--function $f(-q)$ with moduli of orders 1, 3, 5 and 15. For example, he proved that 
 \vspace{0.3cm}
 \newline
If $P:=\dfrac{f(-q^3)f(-q^5)}{q^{1/3}f(-q)f(-q^{15})}$\ \ and \
\ $Q:=\dfrac{f(-q^6)f(-q^{10})}{q^{2/3}f(-q^2)f(-q^{30})}$, then

\begin{equation}\label{S26}
PQ+\frac{1}{PQ}=\lb(\frac{Q}{P}\rb)^3+\lb(\frac{P}{Q}\rb)^3+4.
 \end{equation}
 
 The proof of above equation by using classical methods can be found in \cite{BCB2}. Recently, Mahadeva Naika et al. in \cite{CKSHKBMSM}, have derived a new class of modular identities relating $P$ and $Q_r$, where
 \begin{equation}\label{ex}
P:=\frac{q^{1/3}f(-q)f(-q^{15})}{f(-q^3)f(-q^5)} \ \,\,\textrm{and}\,\,\ Q_{r}:=\frac{q^{r/3}f(-q^r)f(-q^{15r})}{f(-q^{3r})f(-q^{5r})},
\end{equation}
for {$r\in$ \{3, 4, 5, 7\} and using these modular relations they have explicitly evaluated several new cubic class invariants and cubic singular moduli.

\noindent The main goal of this paper is to establish new $P$--$Q$ modular relations that involve $\varphi(q)$ and $\psi(-q)$ with moduli of orders 1, 3, 5 and 15, which are not recorded by Ramanujan in his notebooks and also in his lost notebook. We use these modular relations to evaluate several new explicit evaluations of ratios of theta--functions $\varphi(q)$ and $\psi(-q)$, also a continued fraction of order 12.   

Ramanujan in his notebooks listed several explicit evaluations of $\varphi(e^{-\pi\sqrt{n/k}})$  for few rationals $n$ and $k$. Instigated by the works of Ramanujan, Jinhee Yi in \cite{jy2}, introduced two parameterizations $h_{k,n}$ and $h'_{k,n}$  involving the theta-function $\varphi(q)$ and $\varphi(-q)$ as:
\begin{equation}\label{hkn}
h_{k,n}:=\frac{\varphi(e^{-\pi\sqrt{n/k}})}{k^{1/4}\varphi(e^{-\pi \sqrt{nk}})}
\end{equation}
and
\begin{equation}\label{hkkn}
h'_{k,n}:=\frac{\varphi(-e^{-\pi\sqrt{n/k}})}{k^{1/4}\varphi(-e^{-\pi \sqrt{nk}})}.
\end{equation}
Yi systematically studied several properties of $h_{k,n}$ and $h'_{k,n}$, also explicitly evaluated the parameters for different positive rational values of $n$ and $k$. Motivated by works of Yi, in \cite{ND2} and \cite{jy4}, the authors have defined two parameters $l_{k,n}$ and $l'_{k,n}$  involving the theta-function $\psi(-q)$ and $\psi(q)$ as follows:
\begin{equation}\label{lkn}
l_{k,n}:=\frac{\psi(-e^{-\pi\sqrt{n/k}})}{k^{1/4}e^{-\frac{(k-1)\pi}{8}\sqrt{n/k}}\psi(-e^{-\pi\sqrt{nk}})}
\end{equation}
and
\begin{equation}\label{lkkn}
l'_{k,n}:=\frac{\psi(e^{-\pi\sqrt{n/k}})}{k^{1/4}e^{-\frac{(k-1)\pi}{8}\sqrt{n/k}}\psi(e^{-\pi\sqrt{nk}})}.
\end{equation}

They have also established several properties and some explicit evaluations of $l_{k,n}$ and  $l'_{k,n}$ for different positive rational values of $n$ and $k$.

The work is organized as follows. In Section \ref{S22}, we collect some relevant identities which are used in the subsequent sections. New $P$--$Q$ ``mixed'' modular identities of degree 1, 3, 5 and 15  involving $\varphi(q)$ and $\psi(-q)$ are established in Section \ref{S3}. Several new  explicit evaluations of quotients of $\varphi(q)$ and $\psi(-q)$ are evaluated in Sections \ref{S4} and \ref{S5}. In Section \ref{S6} explicit evaluations of a continued fraction of order 12 are established.

\section{Preliminary results}\label{S22}
To begin with we shall list equations which are helpful in proving our main results. For concise we set
\begin{equation}
B_{r}:= q^{r/3}\frac{f(-q^r)f(-q^{15r})}{f(-q^{3r})f(-q^{5r})}.
\end{equation}
\begin{lemma}\label{l8}\cite[Ch. 16, Entry 24 (ii) \& (iv), p. 39]{BCB1} We have
	\begin{eqnarray}
	&f^3(-q)=\varphi^2(-q)\psi(q),
	\\&f^3(-q^2)=\varphi(-q)\psi^2(q).
	\end{eqnarray}
	\newpage
\end{lemma} \begin{lemma} The following identity holds 
\begin{equation}\begin{split}\label{S21}
	B_1^3+B_1^2B_2^2+B_2^3=B_1B_2.
	\end{split}\end{equation}
\end{lemma}
\begin{proof}
	Consider the identity \eqref{S26} which is recorded by Ramanujan \cite[Ch. 25, Entry 59, p. 214]{BCB2}, factoring this identity, we get
	\begin{equation}\label{S211}
	(B_1^3-B_1B_2+B_1^2B_2^2+B_2^3)(B_1^3+B_1B_2-B_1^2B_2^2+B_2^3)=0.
	\end{equation}
	 We find that the first factor of \eqref{S211} vanishes and the second factor does not vanish for the sequence $\displaystyle \{q_n\}=\lb\{\frac{1}{1+n}\rb\}$.
	 Hence, first factor is identically equal to zero on $|q|<1$. This completes the proof.
\end{proof}
\begin{lemma}\cite{CKSHKBMSM}
	\text{If}\ \ $P:=B_1B_{3}$\ {and} \ $Q:=\dfrac{B_1}{B_{3}},$\ \ \  then
	\begin{equation}\begin{split}\label{S2}
	&\lb(P^3+\frac{1}{P^3}\rb)\lb(46+Q^6+\frac{1}{Q^6}\rb)-9\lb(\sqrt{P^9}-\frac{1}{\sqrt{P^9}}\rb)\lb(\sqrt{Q^3}+\frac{1}{\sqrt{Q^3}}\rb)\\&
	+9\lb(\sqrt{P^3}-\frac{1}{\sqrt{P^3}}\rb)\lb[\sqrt{Q^9}+\frac{1}{\sqrt{Q^9}}+2\lb(\sqrt{Q^3}+\frac{1}{\sqrt{Q^3}}\rb)\rb]=P^6+\frac{1}{P^6}\\&
	+Q^6+\frac{1}{Q^6}+92.
	\end{split}\end{equation}
\end{lemma}

\begin{lemma}\cite{CKSHKBMSM}
	\text{If}\ \ $P:=B_1B_{4}$\ \ {and} \ \ $Q:=\dfrac{B_1}{B_{4}},$\ \ \  then
	\begin{equation}\begin{split}\label{S37}
	&Q^3+\frac{1}{Q^3}+5\lb(Q^2+\frac{1}{Q^2}\rb)+\lb(\sqrt{P^3}
	-\frac{1}{\sqrt{P^3}}\rb)\lb(\sqrt{Q}+\frac{1}{\sqrt{Q}}\rb)\\&+14\lb(Q+\frac{1}{Q}\rb)+18=0.
	\end{split}\end{equation}
\end{lemma}

\begin{lemma}\cite{CKSHKBMSM}
	\text{If}\ \ $P:=B_1B_{5}$\ \ {and} \ \ $Q:=\dfrac{B_1}{B_{5}},$\ \ \  then
	\begin{equation}\begin{split}\label{S38}
	&\lb(Q^3+\frac{1}{Q^3}\rb)\lb(1-P-\frac{1}{P}+P^2+\frac{1}{P^2}\rb)+6\lb(P+\frac{1}{P}\rb)-21\lb(P^2+\frac{1}{P^2}\rb)\\&
	-5\lb(\sqrt{Q^3}+\frac{1}{Q^3}\rb)\lb[\sqrt{P}-\frac{1}{\sqrt{P}}-2\lb(\sqrt{P^3}-\frac{1}{\sqrt{P^3}}\rb)+\sqrt{P^5}-\frac{1}{\sqrt{P^5}}\rb]
	\\&+11\lb(P^3+\frac{1}{P^3}\rb)-\lb(P^4+\frac{1}{P^4}\rb)+8=0.
	\end{split}\end{equation}
\end{lemma}

\begin{lemma}\cite{CKSHKBMSM}
	\text{If}\ \ $P:=B_1B_{7}$\ \ {and} \ \ $Q:=\dfrac{B_1}{B_{7}},$\ \ \  then
	\begin{equation}\begin{split}\label{S39}
	&Q^4+\frac{1}{Q^4}+7\lb(Q^3+\frac{1}{Q^3}\rb)+35\lb(Q^2+\frac{1}{Q^2}\rb)+112\lb(Q+\frac{1}{Q}\rb)-\lb(P^3+\frac{1}{P^3}\rb)\\&
	-7\lb(\sqrt{P^3}-\frac{1}{\sqrt{P^3}}\rb)\lb[\sqrt{Q^3}+\frac{1}{\sqrt{Q^3}}+4\lb(\sqrt{Q}+\frac{1}{\sqrt{Q}}\rb)\rb]+168=0.
	\end{split}\end{equation}
\end{lemma}
\begin{lemma}\cite{CA3}
	If\,\, $P:=\dfrac{\varphi^4(q)}{\varphi^4(q^3)}$\,\, and \,\, $Q:=\dfrac{\psi^4(-q)}{q\psi^4(-q^{3})}$,\,\, then
	\begin{equation}\label{D3}
	P+PQ=9+Q.
	\end{equation}
\end{lemma}

\begin{lemma}\cite{CA3}
	If\,\, $P:=\dfrac{\varphi^2(q)}{\varphi^2(q^5)}$\,\, and \,\, $Q:=\dfrac{\psi^2(-q)}{q\psi^2(-q^{5})}$,\,\, then
	\begin{equation}\label{D5}
	P+PQ=5+Q.
	\end{equation}
\end{lemma}
\begin{lemma}\cite{NDB1}
	If\,\, $P:=\dfrac{\varphi(q)}{\varphi(q^3)}$\,\, and \,\, $Q:=\dfrac{\varphi(q^5)}{\varphi(q^{15})}$,\,\, then
	\begin{equation}\label{NDB5}
	(PQ)^2+\dfrac{9}{(PQ)^2}=\left(\frac{Q}{P}\right)^3+5\left(\frac{Q}{P}\right)^2+5\left(\frac{P}{Q}\right)^2+5\left(\frac{Q}{P}-\frac{P}{Q}\right)-\left(\frac{P}{Q}\right)^3.
	\end{equation}
\end{lemma}
\begin{lemma}\cite[Ch. 25, Entry 67, p. 235]{BCB2}
	If\,\, $P:=\dfrac{\varphi(q)}{\varphi(q^5)}$\,\, and \,\, $Q:=\dfrac{\varphi(q^5)}{\varphi(q^{15})}$,\,\, then
	\begin{equation}\label{SR3}
	PQ+\dfrac{5}{PQ}=\left(\frac{Q}{P}\right)^2+3\left(\frac{Q}{P}\right)+3\left(\frac{P}{Q}\right)-\left(\frac{P}{Q}\right)^2.
	\end{equation}
\end{lemma}
\begin{lemma}\label{l7}
	If  $P:=\dfrac{\varphi(q^3)\varphi(q^5)}{\varphi(q)\varphi(q^{15})}$ and $Q:=\dfrac{\psi(-q^3)\psi(-q^5)}{q\psi(-q)\psi(-q^{15})}$, then
	\begin{equation}\label{MSMCKSHM}
	Q=\frac{1+P}{1-P}.
	\end{equation}
\end{lemma}
\begin{proof}
	Changing $q$ to $-q$ in the equations \eqref{S21} and cubing the resultant equation on both sides, we have
	\begin{equation}\label{phi1}
	u^3v^3+v^9-6u^3v^6+6u^6v^3-u^9+u^6v^6=0, 
	\end{equation}
	where\,\,$u:=B_1(-q) \,\,\text{and}\,\,v:=B_2.$
	
	By lemma \ref{l8}, it is facile to observe that $u^3=P^2Q$ and $v^3=PQ^2$. By factoring the above equation \eqref{phi1}, we obtain
	\begin{equation}\label{r3}
	(PQ-P+1+Q)(P^2Q^2+P^2Q+P^2-PQ^2-3PQ+P+Q^2-Q+1)=0.
	\end{equation}\label{phi2}
	Observe the first factor of \eqref{r3} vanishes and second factor does not vanish for the sequence $\{q_n\}=\lb\{\frac{1}{1+n}\rb\}$. Hence, first factor is identically equal to zero on $|q|<1$. This completes the proof.
\end{proof}

\section{ $P$--$Q$ ``mixed'' modular equations}\label{S3}
 In this section, we establish several new ``mixed'' modular equations involving Ramanujan's theta--function $\varphi(q)$. Throughout this section, we set
\begin{equation}
A_{r}:= \frac{\varphi(q^r) \varphi(q^{15r})}{\varphi(q^{3r}) \varphi(q^{5r})}
\,\,\,\,\ \textrm {and} \,\,\,\,\ C_{r}:=\dfrac{q^r\psi(-q^r)\psi(-q^{15r})}{\psi(-q^{3r})\psi(-q^{5r})} .
\end{equation}
\begin{theorem}
	\text{If}\ \ $P:=A_{1}A_{2}$\ \ {and} \ \ $Q:=\dfrac{A_{1}}{A_{2}},$\ \ \  then
	\begin{equation}\begin{split}
	\label{S01}
&	Q^2+\frac{1}{Q^2}+P^2+\frac{1}{P^2}+6\left(P+\frac{1}{P}\right)\left[Q+\frac{1}{Q}-4\right]-8\left(Q+\frac{1}{Q}\right)\\&+4\left(\sqrt P-\frac{1}{\sqrt P}\right)\left[\left(\sqrt{Q^3}+\frac{1}{\sqrt{Q^3}}\right)-4\left(\sqrt{Q}+\frac{1}{\sqrt{Q}}\right)\right]\\&+4\left(\sqrt{P^3} -\frac{1}{\sqrt{P^3}}\right)\left(\sqrt{Q}+\frac{1}{\sqrt{Q}}\right)+36=0.
	\end{split} \end{equation}
\end{theorem}

\begin{proof}
	From lemma \ref{l8} and lemma \ref{l7}, we have  
	\begin{equation}\label{S311}
	\left(\dfrac{f(-q^2)f(-q^{30})}{f(-q^6)f(-q^{10})}\right)^3=u\left(\dfrac{u-1}{u+1}\right)^2,\,\,\text{where}\,\,\,u=\dfrac{\varphi(q)\varphi(q^{15})}{\varphi(q^3)\varphi(q^5)}.
	\end{equation}
	Cubing the equation \eqref{S21}, we deduce that 
	\begin{equation}\label{S3312}
     B_1^3B_2^3=B_1^9+6B_1^6B_2^3+6B_1^3B_2^6+B_2^9+B_1^6B_2^6.
	\end{equation}
	Invoking \eqref{S311} in \eqref{S3312}, we get
	\begin{eqnarray}\label{S3113}
	\begin{split}
&(uv+u-1+v)(u^2-u-2uv+uv^2+v^2)(u^2v+u^2-2uv-v+v^2)\\&(1-4v-4u+u^4+v^4+36u^2v^2+16u^2v+16uv^2-24uv+6u^4v^2\\&-16u^3v^2+4u^4v-8u^3v+4uv^4-8uv^3+u^4v^4+4u^4v^3+6u^2v^4\\&-16u^2v^3+4u^3v^4-24u^3v^3-4u^3+6u^2+6v^2-4v^3),\end{split}
	\end{eqnarray}
	where $v=u(q^2).$ 
	
	We find that the last factor of \eqref{S3113} vanishes and other factors does not vanish for the sequence $\{q_n\}=\lb\{\dfrac{1}{1+n}\rb\}$.
	Hence, last factor is identically equal to zero on $|q|<1$. By setting $P:=uv$ and $\displaystyle Q:=\frac{u}{v},$ we arrive at \eqref{S01}. This completes the proof.
\end{proof}
\begin{remark}\label{r2}
	Observe that the equation \eqref{S311} can also be re written as follows:
	\begin{equation}\label{S313}
	\left(\dfrac{f(-q^2)f(-q^{30})}{f(-q^6)f(-q^{10})}\right)^3=u^2\left(\dfrac{u+1}{u-1}\right)^2,\,\,\text{where}\,\,\,u=\dfrac{q\psi(-q)\psi(-q^{15})}{\psi(-q^3)\psi(-q^5)}.
	\end{equation}
	Using \eqref{S313} and adopting the same technique illustrated in Section \ref{S3} one can easily arrive at the modular relations connecting $C_1$ with $C_r$, for $r\in\{2, 3, 5,\,\text{and}\,\,7\}$. For brevity these relations involving $\psi(-q)$ are not included in this article. 
\end{remark}
\begin{theorem}
\text{If}\ \ $P:=A_{1}A_{3}$\ \ {and} \ \ $Q:=\dfrac{A_{1}}{A_{3}},$\ \ \  then
\begin{equation}\begin{split}
\label{S31}
&P^2+\frac{1}{P^2}+2\left(P+\frac{1}{P}\right)+\left(Q^2+\frac{1}{Q^2}\right)=4+3\left(\sqrt{Q}+\frac{1}{\sqrt{Q}}\right)\\&\times\left[\left(\sqrt{P^3}+\frac{1}{\sqrt{P^3}}\right)-2\left(\sqrt{P}-\frac{1}{\sqrt{P}}\right)\right]+\left(\sqrt{Q^3}+\frac{1}{\sqrt{Q^3}}\right)\\&\times\left[\left(\sqrt{P}+\frac{1}{\sqrt{P}}\right)-3\left(\sqrt{P}-\frac{1}{\sqrt{P}}\right)\right].
\end{split} \end{equation}
\end{theorem}
\begin{proof}
	The proof of the equation \eqref{S31} is similar to \eqref{S01}, except that in the place of the equation \eqref{S21}, \eqref{S2} is used.
\end{proof}
\begin{theorem}
	\text{If}\ \ $P:=A_{1}A_{4}$\ \ {and} \ \ $Q:=\dfrac{A_{1}}{A_{4}},$\ \ \  then
	\begin{equation}\begin{split}
	\label{S411}
	&Q^4+\frac{1}{Q^4}-96\left(Q^3 +\frac{1}{Q^3}\right)+1488\left(Q^2+\frac{1}{Q^2}\right)-3522\left(Q+\frac{1}{Q}\right)+6692\\&P^4+\frac{1}{P^4}+\left(P^3+\frac{1}{P^3}\right)\left[28\left(P+\frac{1}{P}\right)-256\right]-\left(P^2+\frac{1}{P^2}\right)\left[96\left(P+\frac{1}{P}\right)\right.\\&\left.-70\left(P^2+\frac{1}{P^2}\right)-976\right]+\left(P+\frac{1}{P}\right)\left[1064\left(P+\frac{1}{P}\right)-576\left(P^2+\frac{1}{P^2}\right)\right.\\&\left.+28\left(P^3+\frac{1}{P^3}\right)-256\right]+\left(\sqrt P-\frac{1}{\sqrt P}\right) \left[8\left(\sqrt{Q^7}+\frac{1}{\sqrt{Q^7}}\right)-336\left(\sqrt{Q^5}+\frac{1}{\sqrt{Q^5}}\right)\right.\\&\left.+1904\left(\sqrt{Q^3}+\frac{1}{\sqrt{Q^3}}\right)+704\left(\sqrt{Q}+\frac{1}{\sqrt{Q}}\right)\right]+\left(\sqrt {P^3}-\frac{1}{\sqrt {P^3}}\right) \left[56\left(\sqrt{Q^5}+\frac{1}{\sqrt{Q^5}}\right)\right.\\&\left.-576\left(\sqrt{Q^3}+\frac{1}{\sqrt{Q^3}}\right)-592\left(\sqrt{Q}+\frac{1}{\sqrt{Q}}\right)\right]+\left(\sqrt {P^5}-\frac{1}{\sqrt {P^5}}\right) \left[56\left(\sqrt{Q^3}+\frac{1}{\sqrt{Q^3}}\right)\right.\\&\left.+144\left(\sqrt{Q}+\frac{1}{\sqrt{Q}}\right)\right]+\left(\sqrt {P^7}-\frac{1}{\sqrt {P^7}}\right) \left(\sqrt{Q}+\frac{1}{\sqrt{Q}}\right)=0.
	\end{split} \end{equation}
\end{theorem}
\begin{proof}
	The proof of the equation \eqref{S411} is similar to \eqref{S01}, except that in the place of the equation \eqref{S21}, \eqref{S37} is used.
\end{proof}

\begin{theorem}
	\text{If}\ \ $P:=C_{1}C_{4}$\ \ {and} \ \ $Q:=\dfrac{C_{1}}{C_{4}},$\ \ \  then
	\begin{equation*}\label{psi4}\begin{split}
	&Q^4+\frac{1}{Q^4}-\left(Q^3+\frac{1}{Q^3}\right)+7\left(Q^2+\frac{1}{Q^2}\right)+5\left(Q+\frac{1}{Q}\right)+14\\&
	=\left(P^2+\frac{1}{P^2}\right)\left(Q+\frac{1}{Q}\right)+\left(P+\frac{1}{P}\right)\left[6-2\left(Q^2+\frac{1}{Q^2}\right)-7\left(Q+\frac{1}{Q}\right)\right]\\&
	+\left(\sqrt{P}-\frac{1}{\sqrt{P}}\right)\left[\left(\sqrt{Q^7}+\frac{1}{\sqrt{Q^7}}\right)-2\left(\sqrt{Q^5}+\frac{1}{\sqrt{Q^5}}\right)-7\left(\sqrt{Q^3}+\frac{1}{\sqrt{Q^3}}\right)\right.
	\end{split}\end{equation*}
	\begin{equation}\begin{split}
	&\left.-12\left(\sqrt{Q}+\frac{1}{\sqrt{Q}}\right)\right]+\left(\sqrt{P^3}-\frac{1}{\sqrt{P^3}}\right)\left[3\left(\sqrt{Q^3}+\frac{1}{\sqrt{Q^3}}\right)+5\left(\sqrt{Q}+\frac{1}{\sqrt{Q}}\right)\right]
	\end{split}\end{equation}
\end{theorem}
\begin{proof}
	The proof of the equation \eqref{psi4} is similar to \eqref{S01}, except that in the place of the equation \eqref{S21}, \eqref{S37} is used .
\end{proof}
\begin{theorem}
	\text{If}\ \ $P:=A_{1}A_{5}$\ \ {and} \ \ $Q:=\dfrac{A_{1}}{A_{5}},$\ \ \  then
	\begin{equation}\begin{split}
			\label{S51}
			&P^4+\frac{1}{P^4}+\left(P^3+\frac{1}{P^3}\right)\left[5\left(Q+\frac{1}{Q}\right)+14\right]+42=Q^3+\frac{1}{Q^3}-15\left(Q+\frac{1}{Q}\right)\\&+\left(P^2+\frac{1}{P^2}\right)\left[4+10\left(Q+\frac{1}{Q}\right)+5\left(Q^2+\frac{1}{Q^2}\right)+\left(Q^3+\frac{1}{Q^3}\right)\right]\\&+\left(P+\frac{1}{P}\right)\left[6+10\left(Q+\frac{1}{Q}\right)-\left(Q^3-\frac{1}{Q^3}\right)\right]+5\left(\sqrt{P^7} -\frac{1}{\sqrt{P^7}}\right)\left(\sqrt Q+\frac{1}{\sqrt Q}\right)\\&+15\left(\sqrt{P^5} -\frac{1}{\sqrt{P^5}}\right)\left(\sqrt Q+\frac{1}{\sqrt Q}\right)-5\left(\sqrt{P^3} -\frac{1}{\sqrt{P^3}}\right)\left[7\left(\sqrt Q+\frac{1}{\sqrt Q}\right)\right.\\&\left.+5\left(\sqrt{Q^3}+\frac{1}{\sqrt{Q^3}}\right)+\left(\sqrt {Q^5}+\frac{1}{\sqrt{Q^5}}\right)\right]+5\left(\sqrt{P} -\frac{1}{\sqrt{P}}\right)\left[11\left(\sqrt Q+\frac{1}{\sqrt Q}\right)\right.\\&\left.+5\left(\sqrt{Q^3}+\frac{1}{\sqrt{Q^3}}\right)+\left(\sqrt {Q^5}+\frac{1}{\sqrt{Q^5}}\right)\right].
		\end{split} \end{equation}
	\end{theorem}
\begin{proof}
	The proof of the equation \eqref{S51} is similar to \eqref{S01}, except that in the place of the equation \eqref{S21}, \eqref{S38} is used.
\end{proof}
\begin{theorem}
	\text{If}\ \ $P:=A_{1}A_{7}$\ \ {and} \ \ $Q:=\dfrac{A_{1}}{A_{7}},$\ \ \  then
	\begin{equation}\begin{split}
	\label{S71}
	&P^3+\frac{1}{P^3}+14\left(P^2+\frac{1}{P^2}\right)-35\left(P+\frac{1}{P}\right)+42=Q^4+\frac{1}{Q^4}\\& +14\left(Q+\frac{1}{Q}\right)\left[4\left(P+\frac{1}{P}\right)-\left(P^2+\frac{1}{P^2}\right)-6\right]+7\left(Q^2+\frac{1}{Q^2}\right)\\&\times\left(\sqrt P+\frac{1}{\sqrt P}\right)^2+7\left(\sqrt{Q}+\frac{1}{\sqrt{Q}}\right) \left[\left(\sqrt{P}-\frac{1}{\sqrt{P}}\right)-2\left(\sqrt{P^3}-\frac{1}{\sqrt{P^3}}\right)\right.\\&\left.+\left(\sqrt{P^5}-\frac{1}{\sqrt{P^5}}\right)\right]-14\left(\sqrt{Q^5} +\frac{1}{\sqrt{Q^5}}\right)\left(\sqrt P-\frac{1}{\sqrt P}\right)\\&+7\left(\sqrt{Q^3}+\frac{1}{\sqrt{Q^3}}\right)\left[\left(\sqrt{P^3}-\frac{1}{\sqrt{P^3}}\right)-\left(\sqrt{P}-\frac{1}{\sqrt{P}}\right)\right].
	\end{split} \end{equation}
\end{theorem}
\begin{proof}
	The proof of the equation \eqref{S71} is similar to \eqref{S01}, except that in the place of the equation \eqref{S21}, \eqref{S39} is used.
\end{proof}

\section{Explicit Evaluation of ratios of theta--functions $\varphi(q)$}\label{S4}
 In the present section, we explicitly evaluate $h_{k,n}$ by using the modular relations established in Section \ref{S3}. By using the definition of $h_{k,n}$ with $k=3$  and $k=5$ respectively the equations \eqref{NDB5} and \eqref{SR3} respectively takes the form.
 \begin{lemma} The following identities hold for any positive real number $n$:
 	\begin{equation}\begin{split}\label{NDBh5}
 	&3\left(h^2_{3,n}h^2_{3,25n}+\dfrac{1}{(h^2_{3,n}h^2_{3,25n}}\right)=\left(\frac{h_{3,25n}}{h_{3,n}}\right)^3  +5\left(\frac{h_{3,25n}}{h_{3,n}}\right)^2+5\left(\frac{h_{3,n}}{h_{3,25n}}\right)^2\\&+5\left(\frac{h_{3,25n}}{h_{3,n}} -\frac{h_{3,n}}{h_{3,25n}}\right)-\left(\frac{h_{3,n}}{h_{3,25n}}\right)^3
 	\end{split}\end{equation}
 	and
	\begin{equation}\label{SRh3}
	\sqrt 5h_{5,n}h_{5,9n}+\dfrac{\sqrt 5}{h_{5,n}h_{5,9n}}=\left(\frac{h_{5,9n}}{h_{5,n}}\right)^2+3\left(\frac{h_{5,9n}}{h_{5,n}}\right)+3\left(\frac{h_{5,n}}{h_{5,9n}}\right)-\left(\frac{h_{5,n}}{h_{5,9n}}\right)^2.
	\end{equation}
\end{lemma}
 \begin{lemma}\cite{jy2} For all positive real number $k, a, b, c, m, n,$ and $d$, with $ab=cd,$ we have 
 	\begin{align}
 	&h_{a,b}h_{kc,kd}=h_{ka,kb}h_{c,d},\label{jy6}\\
  	&h_{k,n/m}h_{m,nk}=h_{n,mk}.\label{jy7}
  	\end{align}
  \end{lemma}
\begin{theorem} We have
\begin{eqnarray}
h_{3,15}&=&3^{1/4}(\sqrt{5}-2)^{1/4}\lb(\frac{\sqrt{5}-\sqrt{3}}{2}\rb)^{1/2}(1+\sqrt{3})^{1/2}, \label{S41}\\
h_{3,5/3}&=&3^{-1/4}(2+\sqrt{5})^{1/4}\lb(\frac{\sqrt{5}-\sqrt{3}}{2}\rb)^{1/2}(1+\sqrt{3})^{1/2}\label{S44}.
\end{eqnarray}
\end{theorem}
\begin{proof}
	Employing \eqref{hkn} in \eqref{S31} with $n=1/15$ and recalling that $h_{3,n}h_{3,1/n}=1,$ we obtain
	\begin{equation}\label{S45}
	t^2-6t+4=0,\,\,\,\text{where}\,\,\,t:=x-\frac{1}{x}\,\,\text{and}\,\,x:=h_{3,1/15}h_{3,3/5}.
	\end{equation}
	Since $0<t<1,$ we find that
	\begin{equation}\label{S46}
	x-\frac{1}{x}=3-\sqrt 5.
	\end{equation}
	As $x>1,$ on solving \eqref{S46}, we get
	 \begin{equation}\label{S47}
	 h_{3,1/15}h_{3,3/5}=\frac{(\sqrt 5+\sqrt 3)(\sqrt 3-1)}{2}.
	 \end{equation}
	 Now setting $n=1/15$ in \eqref{NDBh5} and using \eqref{S47}, we deduce that
	 \begin{equation}\label{S48}
	 	(3s^2-2-\sqrt 5)(s^2+6-3\sqrt 5)=0, \,\,\,\text{where}\,\,\,s:=h_{3,1/15}h_{3,5/3}.
	 \end{equation}
	 Since $s>1$, we find that
	 \begin{equation}\label{S49}
	 h_{3,1/15}h_{3,5/3}=\frac{(2+\sqrt 5)^{1/2}}{\sqrt 3}.
	 \end{equation}
	  Using \eqref{S47} and \eqref{S49}, we obtain \eqref{S41} and \eqref{S44}.
\end{proof}

\begin{theorem}\label{42} We have
	\begin{eqnarray}
	h_{3,20}&=&(\sqrt 2-1)\left\{(\sqrt 5+2)b\right\}^{1/2},\label{Sh20}\\	
	h_{3,4/5}&=&(\sqrt 2+1)\left\{(\sqrt 5-2)b\right\}^{1/2},\label{Sh201}\\
	h_{5,12}&=&(\sqrt{2}-1)\left\{(\sqrt{5}+2)a\right\}^{1/2},\label{Sh202}\\
	h_{5,4/3}&=&(\sqrt{2}+1)\left\{(\sqrt{5}-2)a\right\}^{1/2}\label{Sh203},\\
	h_{4,15}&=&\sqrt{ab},\label{ji73}\\
	h_{4,5/3}&=&\sqrt{\dfrac{b}{a}},\label{ji74}
	\end{eqnarray}
	\text{where}
	\begin{align*}
	 &a=\sqrt{11689-3696\sqrt{10}}-\sqrt{11688-3696\sqrt{10}},\\
	 &b=\sqrt{760-240\sqrt{10}}-\sqrt{759-240\sqrt{10}}.
	\end{align*}	
\end{theorem}
\begin{proof}
	Employing \eqref{hkn} in \eqref{S411} with $n=1/20$ and recalling that $h_{3,n}h_{3,1/n}=1,$ we obtain
	\begin{equation}\label{pSh20}
	t^2+2+24t-18=0,\,\,\,\text{where}\,\,\,t:=x-\frac{1}{x}\,\,\text{and}\,\,x:=h_{3,1/20}h_{3,4/5}.
	\end{equation}
	Since $t>1$, we have
	\begin{equation}
	x-\frac{1}{x}=4\sqrt{10}-12.
	\end{equation}
	On solving the above equation for $x$ and noting that $x>0$, we obtain
	\begin{equation}\label{pSh21}
	h_{3,1/20}h_{3,4/5}=(\sqrt 2+1)^2(\sqrt 5-2).
	\end{equation}
	 Now setting $n=1/20$ in \eqref{NDBh5} and using \eqref{pSh21}, we deduce that
	 \begin{equation}
	 (s^2+(12\sqrt{10}-40)s+1)(s^2+(40-12\sqrt{10})s+1)=0,\,\,\text{where}\,\,s=h_{3,1/20}h_{3,5/4}.
	 \end{equation}
	 Since $s>1$, we have
	 \begin{equation}\label{pSh214}
	 h_{3,1/20}h_{3,5/4}=\sqrt{760-240\sqrt{10}}+\sqrt{759-240\sqrt{10}}.
	 \end{equation}
	 Using \eqref{pSh214} and \eqref{pSh21}, we arrive at \eqref{Sh20} and \eqref{Sh201}. 
	 
	 Now we proceed to prove $h_{5,12}$ and $h_{5,3/4}$. Setting $a = $5, $b = 12$, $c = 3$, $d = 20$ and $k = 1/4$ in \eqref{jy6} and using the fact $h_{k,n}=h_{n,k},$ we find that
	 \begin{equation}\label{h520}
	 h_{3,20}h_{3,5/4}=h_{5,12}h_{5,3/4}.
	 \end{equation}
	 From \eqref{pSh21} and \eqref{h520}, we have
	 \begin{equation}\label{h5201}
	 h_{5,1/12}h_{5,4/3}=(\sqrt 2+1)^2(\sqrt 5-2).
	 \end{equation}
	 Now, setting $n=1/12$ in the equation \eqref{SRh3} and using \eqref{h5201}, we deduce that 
	 \begin{equation}\label{h5202}
	 s^2+(48\sqrt{10}-154)s+1=0, \,\,\text{where}\,\,s:=h_{5,1/12}h_{5,3/4}.
	 \end{equation}
	 Since $s>1$, on solving the above equation \eqref{h5202}, we get
	 \begin{eqnarray}\label{h5203}
	 h_{5,1/12}h_{5,3/4}=\sqrt{11689-3696\sqrt{10}}+\sqrt{11688-3696\sqrt{10}}.
	 \end{eqnarray}
	 By \eqref{h5203}, \eqref{h5201} and the fact $h_{k,n}h_{k,1/n}=1$, we obtain \eqref{Sh202} and \eqref{Sh203}.
	 
	 Again,	setting $k=3, n=4$ and $m=5$, in \eqref{jy7}, we have
	 	\begin{equation}\label{ji71}
	 	h_{4,15}=h_{3,4/5}h_{5,12}.
	 	\end{equation}
	 	Setting $k=4, n=5$ and $m=3$, in \eqref{jy7}, we have
	 	\begin{equation}\label{ji72}
	 	h_{4,5/3}=h_{5,12}h_{3,1/20}.
	 	\end{equation}
	 	From the equations \eqref{ji71} and \eqref{ji72}, we arrive at \eqref{ji73} and \eqref{ji74}.
\end{proof}

\begin{theorem} We have
	\begin{eqnarray}
	&&h_{5,15}=2^{-1/6}(5^{1/6}-\sqrt{5^{1/3}-2^{2/3}})^{1/2}\left(\frac{-5}{3}+\frac{10^{1/3}}{3}+\frac{10^{2/3}}{3}\right)^{1/2}, \label{S512}\\&&h_{5,5/3}=2^{-1/6}(5^{1/6}+\sqrt{5^{1/3}-2^{2/3}})^{1/2}\left(\frac{-5}{3}+\frac{10^{1/3}}{3}+\frac{10^{2/3}}{3}\right)^{1/2}	\label{S513}.\end{eqnarray}
\end{theorem}
\begin{proof}
	Employing \eqref{hkn} in \eqref{S51} with $n=1/15$ and recalling that $h_{5,n}h_{5,1/n}=1,$ we obtain
	\begin{equation}\label{S4511}
	t^3+8t-2t^2=4,\,\,\,\text{where}\,\,\,t:=x-\frac{1}{x}\,\,\text{and}\,\,x=h_{5,1/15}h_{5,5/3}.
	\end{equation}
	Since $t>0,$ we find that
	\begin{equation}\label{S4512}
	h_{5,1/15}h_{5,5/3}=2^{-1/3}(5^{1/6}-\sqrt{5^{1/3}-2^{2/3}}).
	\end{equation}
	Now setting $n=1/15$ in \eqref{SRh3} and using \eqref{S4512}, we deduce that
	\begin{equation}\label{S4513}
	(3s+5-10^{1/3}-10^{2/3})(15s-5-2(10)^{2/3}-5(10)^{1/3})=0, 
	\end{equation}
	where\,\,\,$s=h_{5,1/15}h_{5,3/5}$. Since $s>1$, we find that
	\begin{equation}\label{S4514}
	h_{5,1/15}h_{5,3/5}=\frac{1}{3}+2\frac{10^{2/3}}{15}+\frac{10^{1/3}}{3}.
	\end{equation}
	Using \eqref{S4514} and \eqref{S4512}, we obtain \eqref{S512} and \eqref{S513}.
\end{proof}

\begin{theorem}\label{44} We have
	\begin{eqnarray}
	h_{3,35}&=&2^{-1/2}[(9-4\sqrt 5) (\sqrt{21}+2\sqrt{5})]^{{1}/{4}}[(\sqrt 7-\sqrt{5})(\sqrt 5+\sqrt{3})]^{{1}/{2}}, \label{S735}
	\\h_{3,7/5}&=&2^{-1/2}[(9-4\sqrt 5) (\sqrt{21}+2\sqrt{5})]^{{1}/{4}}[(\sqrt 7+\sqrt{5})(\sqrt 5-\sqrt{3})]^{{1}/{2}},\label{S736}\\
	h_{5,21}&=&2^{-1}\left\{(3\sqrt 3-5) (3+\sqrt{7})(\sqrt 7-\sqrt{5})(\sqrt 5+\sqrt{3})\right\}^{{1}/{2}}, \label{S521}\\
	h_{5,7/3}&=&2^{-1}\left\{(3\sqrt 3- 5) (3+\sqrt{7})(\sqrt 7+\sqrt{5})(\sqrt 5-\sqrt{3})\right\}^{{1}/{2}},\label{S522}\\
		h_{7,15}&=&2^{-1/2}\{(9-4\sqrt 5)(\sqrt{21}-2\sqrt{5})\}^{1/4}\{(3\sqrt{3}-5)(3+\sqrt
		7)\}^{1/2},\label{ji75}\\
		h_{7,5/3}&=&2^{1/2}\{(9+4\sqrt 5)(\sqrt{21}+2\sqrt{5})\}^{1/4}\{(3\sqrt{3}-5)(3+\sqrt
		7)\}^{1/2}.\label{ji76}
	\label{S738}\end{eqnarray}
\end{theorem}
\begin{proof}
	Employing \eqref{hkn} and \eqref{S71} with $n=1/35$ and recalling that $h_{3,n}h_{3,1/n}=1,$ we obtain
	\begin{equation}\label{S7h1}
	t^2+2-10t+2=0,\,\,\text{where}\,\,t:=x-\frac{1}{x}\,\,\text{and}\,\,x:=h_{3,1/35}h_{3,7/5}.
	\end{equation}
	On solving the above equation for $t$ and observe that $0<t<1$, we have
	\begin{equation}\label{S7h2}
	x-\frac{1}{x}=5-\sqrt{21.}
	\end{equation}
	Again, solving the above equation for $x$ and notice that $x>0$, we get
	\begin{equation}\label{S7h3}
	h_{3,1/35}h_{3,7/5}=\frac{(\sqrt 7+\sqrt 5)(\sqrt 5-\sqrt 3)}{2}.
		\end{equation}
	Set $n=1/35$ in \eqref{NDBh5} and with the help of \eqref{S7h3}, we deduce that 
	\begin{eqnarray}\label{S7h4}
	(s^2-9\sqrt{21}+18\sqrt 5-4\sqrt{105}-40)(s^2-9\sqrt{21}-18\sqrt 5+4\sqrt{105}+40)=0,
	\end{eqnarray}
	where $s=h_{3,35}h_{3,7/5}.$ 
	Since $0<s<1$, we get 
	\begin{equation}\label{S7h5}
	h_{3,35}h_{3,7/5}=\sqrt{(9-4\sqrt 5)(\sqrt{21}+2\sqrt 5)}.
	\end{equation}
	Using \eqref{S7h3} and \eqref{S7h5}, we obtain \eqref{S735} and \eqref{S736}.

	 Now we proceed to prove $h_{5,21}$ and $h_{5,7/3}$. Setting $a = 5$, $b = 21$, $c = 3$, $d = 35$ and $k = 1/7$ in \eqref{jy6} and using the fact $h_{k,n}=h_{n,k},$ we find 	 \begin{equation}\label{h521}
	 h_{3,35}h_{3,5/7}=h_{5,12}h_{5,3/7}.
	 \end{equation}
	 From \eqref{S7h3} and \eqref{h521}, we have
	 \begin{equation}\label{h5211}
	 h_{5,21}h_{5,3/7}=\frac{(\sqrt 7-\sqrt 5)(\sqrt 5+\sqrt 3)}{2}.
	 \end{equation}
	 Now set $n=1/21$ in \eqref{SRh3} and using \eqref{h5211}, we find
	 \begin{equation}\label{h5212}
	(2s-9\sqrt 3+5\sqrt 7-15+3\sqrt{21})(2s-9\sqrt 3+5\sqrt 7+15-3\sqrt{21})=0,
	 \end{equation}
	 where $s:=h_{5,1/21}h_{5,3/7}$. Note that $s>1$, we find that
	 \begin{equation}\label{h5213}
	 h_{5,1/21}h_{5,3/7}=\frac{(3\sqrt 3-5)(3-\sqrt 7)}{2}.
	 \end{equation}
	 By \eqref{h5211} and \eqref{h5213} , we obtain \eqref{S521} and \eqref{S522}.
	 
	 Again Setting $k=3, n=7$ and $m=5$, in \eqref{jy7}, we have
	 \begin{equation}\label{ji77}
	 h_{7,15}=h_{3,7/5}h_{5,21}.
	 \end{equation}
	 Setting $k=7, n=5$ and $m=3$, in \eqref{jy7}, we have
	 \begin{equation}\label{ji78}
	 h_{7,5/3}=h_{5,21}h_{3,1/35}.
	 \end{equation}
	 From the equations \eqref{ji77} and \eqref{ji78}, we arrive at \eqref{ji73} and \eqref{ji74}.
\end{proof}

\section{Explicit Evaluation of ratios of theta--functions $\psi(-q)$}\label{S5}
In this section, we explicitly evaluate $l_{k,n}$ with the aid of following lemmas.

\begin{lemma}\cite{ ND2,jy4} For all positive real number $k, a, b, c, m, n,$ and $d$, with $ab=cd,$ we have 
	\begin{align}
	&l_{a,b}l_{kc,kd}=l_{ka,kb}l_{c,d},\label{ljy6}\\
	&l_{k,n/m}l_{m,nk}=l_{n,mk}.\label{ljy7}
	\end{align}
\end{lemma}

\begin{lemma}For all positive real number $n,$ we have 
	\begin{align}
	&h^4_{3,n}+3h^4_{3,n}l^4_{3,n}=3+l^4_{3,n},\label{hl3}\\
	&h^2_{5,n}+\sqrt{5}h^2_{5,n}l^2_{5,n}=\sqrt{5}+l^2_{5,n}.\label{hl5}
	\end{align}
\end{lemma}
\begin{proof}
	Transcribing the equations \eqref{D3} and \eqref{D5} respectively by using the definition of $h_{k,n}$ and $l_{k,n}$ with $k=3$ and $k=5$ respectively, we arrive at \eqref{hl3} and \eqref{hl5}.
\end{proof}
\begin{theorem}\label{l42} We have
	\begin{align}
	l_{3,20}&=\frac{\left\{(3+\sqrt{5})(\sqrt{2}+1)\right\}^{1/2}\left\{\sqrt{394+120\sqrt{10}}+\sqrt{390+120\sqrt{10}}\right\}}{2^{3/4}},\label{lSh20}\\	
	l_{3,4/5}&=\frac{\left\{(3-\sqrt{5})(\sqrt{2}-1)\right\}^{1/2}\left\{\sqrt{394+120\sqrt{10}}+\sqrt{390+120\sqrt{10}}\right\}}{2^{3/4}},\label{lSh201}\\
	l_{5,12}&=\left\{\frac{(3+\sqrt{5})(\sqrt{2}+1)}{2}\right\}^{1/2}\left\{\frac{(\sqrt{46+12\sqrt{10}}+\sqrt{42+12\sqrt{10}})}{2}\right\}^{1/2},\label{lSh202}\\
	l_{5,4/3}&=\left\{\frac{(3-\sqrt{5})(\sqrt{2}-1)}{2}\right\}^{1/2}\left\{\frac{(\sqrt{46+12\sqrt{10}}+\sqrt{42+12\sqrt{10}})}{2}\right\}^{1/2},\label{lSh203}\\
	l_{4,15}&=2^{-3/4}(2\sqrt{3}+\sqrt{10})^{1/2}(2\sqrt{3}+1)^{1/2}(5\sqrt{2}+4\sqrt{3})^{1/4}(\sqrt{5}+\sqrt{3})^{1/4},\label{lji73}\\
	l_{4,5/3}&=2^{-3/4}(2\sqrt{3}+\sqrt{10})^{1/2}(2\sqrt{3}+1)^{1/2}(5\sqrt{2}-4\sqrt{3})^{1/4}(\sqrt{5}-\sqrt{3})^{1/4}.\label{lji74}
	\end{align}
\end{theorem}
\begin{proof}
	Using the equation \eqref{hl3} and respective values of $h_{3,n}$ for $n=20, 4/5$, we arrive at \eqref{lSh20} and \eqref{lSh201} respectively. To prove \eqref{lSh202} and \eqref{lSh203}, we use \eqref{hl5} and respective values of $h_{5,n}$ for $n=12, 4/3$ respectively.
	
	Now we proceed to prove \eqref{lji73} and \eqref{lji74}. Setting $k=3, n=4$ and $m=5$, in \eqref{ljy7}, we have
	\begin{equation}\label{lji77}
	l_{4,15}=l_{3,4/5}l_{5,12}.
	\end{equation}
	Again setting $k=4, n=5$ and $m=3$, in \eqref{ljy7}, we have
	\begin{equation}\label{lji78}
	l_{4,5/3}=l_{5,12}l_{3,1/20}.
	\end{equation}
	From the equations \eqref{lji77} and \eqref{lji78}, we arrive at \eqref{lji73} and \eqref{lji74}.
\end{proof}
\begin{theorem}\label{th1} We have
	\begin{align}
	l_{3,15}&=3^{1/4}(\sqrt{5}+2)^{1/4}\lb(\frac{\sqrt{5}+\sqrt{3}}{2}\rb)^{1/2}(\sqrt{3}+1)^{1/2}, \label{lS41}\\
	l_{3,5/3}&=3^{-1/4}(\sqrt{5}-2)^{1/4}\lb(\frac{\sqrt{5}+\sqrt{3}}{2}\rb)^{1/2}(\sqrt{3}+1)^{1/2}\label{lS44},\\
	l_{3,35}&=2^{-1/2}[(9+4\sqrt 5) (\sqrt{21}+2\sqrt{5})]^{{1}/{4}}[(\sqrt 7+\sqrt{5})(\sqrt 5+\sqrt{3})]^{{1}/{2}},\label{lS735}\\
	l_{3,7/5}&=2^{-1/2}[(9+4\sqrt 5) (\sqrt{21}+2\sqrt{5})]^{{1}/{4}}[(\sqrt 7-\sqrt{5})(\sqrt 5-\sqrt{3})]^{{1}/{2}},\\
	l_{5,21}&=2^{-1}\left\{(3\sqrt 3+5) (3+\sqrt{7})(\sqrt 7+\sqrt{5})(\sqrt 5+\sqrt{3})\right\}^{{1}/{2}}, \label{lS521}\\
	l_{5,7/3}&=2^{-1}\left\{(3\sqrt 3+5) (3+\sqrt{7})(\sqrt 7-\sqrt{5})(\sqrt 5-\sqrt{3})\right\}^{{1}/{2}},\label{lS522}\\
	l_{7,15}&=2^{-1/2}\left\{(3\sqrt{3}+5)(3+\sqrt{7})(\sqrt{5}+2)\right\}^{1/2}(\sqrt{21}+2\sqrt{5})^{1/4},\label{lji75}\\
	l_{7,5/3}&=2^{-1/2}\left\{(3\sqrt{3}+5)(3+\sqrt{7})(\sqrt{5}-2)\right\}^{1/2}(\sqrt{21}-2\sqrt{5})^{1/4}.\label{lji76}\end{align}
\end{theorem}
The proof of Theorem \ref{th1} is similar to the proof of the Theorem \ref{l42}, hence we omit the details. 
\section{A Continued fraction of order 12}\label{S6}
A continued fraction $H=H(q)$  for $|q|<1$ is defined by
\begin{equation}\label{h11}
H:=\frac{q(1-q)}{1-q^3}\cfplus\frac{q^3(1-q^2)(1-q^4)}{(1-q^3)(1+q^6)}\cfplus\frac{q^3(1-q^8)(1-q^{10})}{(1-q^3)(1-q^{12})}
\cfplus\dotsb,
\end{equation}
\begin{equation}\label{h12}
H=q\prod_{j=1}^{\infty}\dfrac{(1-q^{12j-1})(1-q^{12j-11})}{(1-q^{12j-5})(1-q^{12j-7})}=\frac{\varphi(q)-\varphi(q^3)}{\varphi(q)+\varphi(q^3)}.
\end{equation}
The continued fraction expression in \eqref{h11} has been derived by Mahadeva Naika et al. \cite{MSMBND}. For a summary of $H(q)$ and for references to other work related to $H(q)$ see \cite{SC}[Ch. 12].

\begin{lemma}For any positive rational $n$, we have
	\begin{align}\label{h13}
	H(e^{-\pi\sqrt{n}})=\frac{3^{1/4}h_{3,3n}-1}{3^{1/4}h_{3,3n}+1}.
	\end{align}	
	\begin{proof}
		Using the definition of $h_{k,n}$ with $k=3$ and \eqref{h12}, we arrive at  \eqref{h13}.
	\end{proof}
\end{lemma}  
We conclude this section, by tabulating few explicit evaluations of $H(e^{-\pi\sqrt{n}})$, for few positive rationals $n$ by using the values of $h_{3,n}$.
 \vspace{0.2cm}
 
\renewcommand{\arraystretch}{2.3}
\begin{tabular}{|c|c|}
	\hline
	$n$ & $H(e^{-\pi\sqrt{n}})$ \\
	\hline
	${1}/{45}$ & $\dfrac{(\sqrt{5}+2)^{1/4}(\sqrt{5}-\sqrt{3})^{1/2}(\sqrt{3}-1)^{1/2}-\sqrt{2}}{(\sqrt{5}+2)^{1/4}(\sqrt{5}-\sqrt{3})^{1/2}(\sqrt{3}-1)^{1/2}+\sqrt{2}}$ \\ \hline
	${5}/{9}$& $\dfrac{(\sqrt{5}+2)^{1/4}(\sqrt{5}-\sqrt{3})^{1/2}(\sqrt{3}+1)^{1/2}-\sqrt{2}}{(\sqrt{5}+2)^{1/4}(\sqrt{5}-\sqrt{3})^{1/2}(\sqrt{3}+1)^{1/2}+\sqrt{2}}$ \\ \hline
	${20}/{3}$ &  $ \dfrac{3^{1/4}\left\{(\sqrt{5}+2)\left(\sqrt{760-240\sqrt{10}}-\sqrt{759-240\sqrt{10}}\right)\right\}^{1/2}-(\sqrt{2}+1)}{3^{1/4}\left\{(\sqrt{5}+2)\left(\sqrt{760-240\sqrt{10}}-\sqrt{759-240\sqrt{10}}\right)\right\}^{1/2}+(\sqrt{2}+1)}$\\
	\hline
	${35}/{3}$ & $\dfrac{3^{1/4}\left\{(\sqrt{7}-\sqrt{5})(\sqrt{5}+\sqrt{3})\right\}^{1/2}-2^{1/2}\left\{(9+4\sqrt{5})(\sqrt{21}-2\sqrt{5})\right\}^{1/2}}{3^{1/4}\left\{(\sqrt{7}-\sqrt{5})(\sqrt{5}+\sqrt{3})\right\}^{1/2}+2^{1/2}\left\{(9+4\sqrt{5})(\sqrt{21}-2\sqrt{5})\right\}^{1/2}}$\\
	\hline
	${1}/{105}$  & $\dfrac{3^{1/4}\left\{(\sqrt{7}+\sqrt{5})(\sqrt{5}-\sqrt{3})\right\}^{1/2}-2^{1/2}\left\{(9-4\sqrt{5})(\sqrt{21}+2\sqrt{5})\right\}^{1/2}}{3^{1/4}\left\{(\sqrt{7}+\sqrt{5})(\sqrt{5}-\sqrt{3})\right\}^{1/2}+2^{1/2}\left\{(9-4\sqrt{5})(\sqrt{21}+2\sqrt{5})\right\}^{1/2}}$
	\\ \hline
\end{tabular}

\end{document}